\documentclass[12pt]{amsart}
\usepackage{amsmath,amssymb,amsthm,amscd}
\newtheorem{thm}{Theorem}
\newtheorem*{thmnonum}{Theorem}
\newtheorem{lem}[thm]{Lemma}

\newtheorem*{rem}{Remark}
\newtheorem*{ack}{Acknowledgements}

\newcommand{\SL}{{\rm SL}}

\renewcommand{\H}{\mathbb{H}}
\newcommand{\Q}{\mathbb{Q}}
\newcommand{\Z}{\mathbb{Z}}
\newcommand{\R}{\mathbb{R}}

\newcommand{\F}{\mathbb{F}}
\newcommand{\Aut}{{\rm Aut}}

\newcommand{\legen}[2]{\genfrac{(}{)}{}{}{#1}{#2}}
\newcommand{\Gen}{{\rm Gen}}

\newcommand{\Ad}{{\rm Ad}}

\newcommand{\ord}{{\rm ord}}

\begin{document}

\title[Quadratic forms and the Petersson inner product]{Integers represented by positive-definite quadratic forms and Petersson inner products}
\author{Jeremy Rouse}
\address{Department of Mathematics and Statistics, Wake Forest University,
  Winston-Salem, NC 27109}
\email{rouseja@wfu.edu}

\begin{abstract}
Let $Q$ be a positive-definite quaternary quadratic form with integer coefficients. We study the problem of giving bounds on the largest positive integer $n$ that is locally represented by $Q$ but not represented. Assuming that
$n$ is relatively prime to $D(Q)$, the determinant of the Gram matrix of $Q$,
we show that $n$ is represented provided that 
\[
 n \gg \max \{ N(Q)^{3/2 + \epsilon} D(Q)^{5/4 + \epsilon}, N(Q)^{2 + \epsilon} D(Q)^{1 + \epsilon} \}.
\]
Here $N(Q)$ is the level of $Q$. We give three other bounds that hold
under successively weaker local conditions on $n$. 

These results are proven by bounding the Petersson norm of the
cuspidal part of the theta series, which is accomplished using an
explicit formula for the Weil representation due to Scheithauer.
\end{abstract}

\maketitle

\section{Introduction and Statement of Results}
\label{intro}

If $Q$ is a positive-definite quadratic form with integer coefficients,
which integers are represented by $Q$? This innocous sounding question has
resulted in some very deep mathematics, and there are still not complete answers to this question. For example, there is no unconditional answer to
the question of which integers are represented by $x^{2} + y^{2} + 7z^{2}$ (see \cite{Kap7} and \cite{Reinke}), and it is unknown whether every odd number can be represented by $x^{2} + xy + 5y^{2} + 2z^{2}$ (see \cite{Kap}, \cite{451}).

The question raised in the first paragraph has a long history, beginning
with Fermat's claimed classification of which integers are sums of two squares.
The first analytic result handling arbitrary positive-definite forms
is due to Tartakowsky \cite{Tart} in 1929. To state the result, we need some
notation.

Let $Q(\vec{x}) = \frac{1}{2} \vec{x}^{T} A \vec{x}$ be a
positive-definite quadratic form in $r$ variables, and assume that $A$
has integer entries and even diagonal entries. Let $D(Q) = \det(A)$ be
the discriminant of $Q$. Define the \emph{level} of $Q$, $N(Q)$, to be
the smallest positive integer $N$ so that $NA^{-1}$ has integer
entries and even diagonal entries.

We say that $Q$ \emph{locally represents} the integer
$n$ if there is a solution to the congruence
$Q(\vec{x}) \equiv n \pmod{p^{k}}$ for all primes $p$ and all
$k \geq 1$. We say that $n$ is \emph{primitively locally represented}
if one can solve $Q(\vec{x}) \equiv n \pmod{p^{k}}$ for all primes $p$
and all $k \geq 1$ with the additional restriction that $p$ does not
divide all the entries of $\vec{x}$. 

\begin{thmnonum}[Tartakowsky, 1929]
Assume that $r \geq 5$. Then every sufficiently large integer $n$ that
is locally represented by $Q$ is represented. If $r = 4$, the same conclusion
holds provided that $n$ is primitively locally represented.
\end{thmnonum}

Tartakowsky's result does not give a bound on the largest locally represented by non-represented integer $n$. Results of this type were given by Watson \cite{Watson}, who proved
that for $r \geq 5$, any locally represented integer $n \gg D(Q)^{5/(r-4) + 1/r}$
is represented by $Q$ if $5 \leq r \leq 9$ (and the same conclusion holds when $n \gg D(Q)$ if $r \geq 10$).

Watson did not make the implied constant in his estimate
explicit. Hsia and Icaza \cite{HsiaIcaza} gave completely explicit
results that prove that any locally represented integer
$n \gg D(Q)^{(n-2)/(n-4) + 2/n}$ is repesented using an algebraic
method. This improves Watson's exponent in the case that $r = 5$ and
$r = 6$.

The case of $r = 4$ variables is more subtle, because it is no longer
true that every sufficiently large locally represented integer is
represented.  Indeed, $x^{2} + y^{2} + 7z^{2} + 7w^{2}$ locally
represents every positive integer, but fails to represent any integer
of the form $3 \cdot 7^{2k}$ (as observed by Watson in
\cite{Watson2}). The first result in this direction was due to
Kloosterman \cite{Kloosterman} using a variant of the circle method.
Refined results in the $r = 4$ case were proven by Fomenko
(\cite{Fom}, \cite{Fom2}) by considering the theta series of $Q$ as a
weight $2$ modular form and using the Petersson inner product.  In
\cite{SP}, Schulze-Pillot gave a completely explicit bound in the case
that $r = 4$, showing that if $n$ is primitively locally represented
and $n \gg N(Q)^{14+\epsilon}$, then $n$ is represented by
$Q$. (Schulze-Pillot obtains the better bounds
$n \gg N(Q)^{9 + \epsilon}$ and $n \gg N(Q)^{4 + \epsilon}$ under the
assumptions that $N(Q)$ is squarefree, and $D(Q)$ is squarefree,
respectively.)

In 2006, Browning and Dietmann \cite{BrowningDietmann} applied a modern
form of the Hardy-Littlewood circle method and improved the exponents
in Hsia and Icaza's work in the cases that $r = 7, 8$ and $9$. Also, they
prove that in the $r = 4$ case that if stronger local conditions on
$n$ are chosen than any $n \gg \det(Q)^{2} \|Q\|^{8 + \epsilon}$
is represented, where $\|Q\|$ is the height of $Q$ (the absolute value
of the largest coefficient). The height $\|Q\|$ is between
$\det(Q)^{1/4}$ and $\det(Q)$ (depending on the successive minima of
the corresponding lattice). The local condition in question is referred
to by the authors as the `strong local solubility condition' and is the condition that there is some $\vec{x} \in (\Z/p\Z)^{4}$ so that
$Q(\vec{x}) \equiv n \pmod{p^{r}}$ with $p \nmid A\vec{x}$ (and here
$r = 1$ if $p > 2$ and $r = 3$ if $p = 2$).

In 2014, the author \cite{451} proved a stronger result in a more specific case.
If $r = 4$ and $\det(Q)$ is a fundamental discriminant, then every
locally represented integer $n \gg D(Q)^{2 + \epsilon}$ is represented. The
bound given in \cite{451} is ineffective (and is related to the possible
presence of an $L$-function with a Siegel zero) and not explicit. However,
it is amenable to explicit computations and it was used to prove that $x^{2} + 3y^{2} + 3yz + 3yw + 5z^{2} + zw + 34w^{2}$, which has $D(Q) = N(Q) = 6780$, represents all odd positive integers.

The focus of the present paper is to handle arbitrary quaternary
quadratic forms by bounding the Petersson norm of the cuspidal
part of the theta series. (The quaternary case has received a lot of
attention recently. See \cite{Bhar}, \cite{BH}, \cite{451},
\cite{REU2016}, and \cite{Coprime3}.)

The method we use relies on the explicit
formula of Scheithauer \cite{Scheithauer} for the Weil
representation. This yields a formula for the Fourier expansion of the
theta series at each cusp. Our main result is the following.

\begin{thm}
\label{main}
Suppose that $Q$ is a primitive, positive-definite, integer-valued quaternary
quadratic form.
\begin{itemize}
\item If $n$ is relatively prime to $D(Q)$ and is locally represented by $Q$, then $n$ is represented provided
\[
  n \gg \max \{ N(Q)^{3/2 + \epsilon} D(Q)^{5/4 + \epsilon},
  N(Q)^{2 + \epsilon} D(Q)^{1 + \epsilon} \}.
\]
\item If $n$ satisfies the strong local solubility condition for $Q$, then
$n$ is represented by $Q$ provided
\[
  n \gg \max \{ N(Q)^{5/4+\epsilon} D(Q)^{5/4 + \epsilon}, N(Q)^{3+\epsilon}
  D(Q)^{1+\epsilon} \}.
\]
\item If $n$ is primitively locally represented by $Q$,
then $n$ is represented provided 
\[
  n \gg \max \{ N(Q)^{5/2 + \epsilon} D(Q)^{9/4 + \epsilon},
N(Q)^{3 + \epsilon} D(Q)^{2 + \epsilon} \}.
\]
\item If $n$ is any positive integer that is locally represented by $Q$,
$n$ is not represented by $Q$ and
\[
  n \gg \max \{ N(Q)^{9/2 + \epsilon} D(Q)^{5/4 + \epsilon},
N(Q)^{5 + \epsilon} D(Q)^{1 + \epsilon} \},
\]
then there is an anisotropic prime $p$ so that $p^{2} | n$ and $np^{2k}$ is not represented by $Q$ for any $k \geq 0$.
\end{itemize}
\end{thm}

\begin{rem}
  The result stated above in the case that $\gcd(n,N(Q)) = 1$ is
  ineffective and relies on lower bounds for $L(\Ad^{2} f, 1)$ where
  $f$ is a newform. When $f$ has CM, the $L$-function
  $L(\Ad^{2} f, s)$ has a Dirichlet $L$-function as a factor and we
  cannot rule out the possibility that this $L$-function has a Siegel
  zero. This result can be made effective for a given $Q$ at the cost
  of a finite computation.

  In contrast, the results in the second, third and fourth cases are
  effective.  While it is not especially difficult to make these
  results completely explicit, we will not take that step. Examples
  suggest that the resulting bounds would not be small enough to be
  useful in concrete applications.
\end{rem}

\begin{rem}
  In the case that $N(Q) = D(Q)$ and the successive minima of $Q$ are
  of the same magnitude, we obtain that any $n \gg D(Q)^{4+\epsilon}$
  which satisfies the strong local solubility condition for $Q$ is
  represented. This is the same quality result as obtained by Browning
  and Dietmann \cite{BrowningDietmann}. In cases where the successive
  minima of $Q$ are of different orders of magnitude or $N(Q)$ is
  substantially smaller than $D(Q)$, our result improves on Browning
  and Dietmann's.
\end{rem}

In Section~\ref{back} we define notation and review Scheithauer's
formula and elements of the theory of quadratic forms. In
Section~\ref{eis} we bound the coefficients of the Eisenstein series
portion of the theta series. In Section~\ref{pet} we bound the
Petersson norm of the cusp form component and in Section~\ref{last}
we conclude our main results.

\begin{ack}
  The author would like to thank Rainer Schulze-Pillot and Katherine
  Thompson for helpful conversations related to this work. The author would
also like to thank Gergely Harcos and Valentin Blomer for their insight
into a question asked on MathOverflow related to this work (see {\tt http://mathoverflow.net/questions/256576/}).
\end{ack}

\section{Background and notation}
\label{back}

A quadratic form in $r$ variables $Q(\vec{x})$ is integer-valued if it
can be written in the form
$Q(\vec{x}) = \frac{1}{2} \vec{x}^{T} A \vec{x}$, where $A$ is a
symmetric $r \times r$ matrix with integer entries, and even diagonal
entries. The matrix $A$ is called the Gram matrix of $Q$.  The
quadratic form $Q$ is called positive-definite if $Q(\vec{x}) \geq 0$
for all $\vec{x} \in \R^{r}$ with equality if and only if
$\vec{x} = \vec{0}$. We let $D(Q) = \det(A)$, and we define $N(Q)$ to
be the level of $Q$, which is the smallest positive integer $N$ so
that $N A^{-1}$ has integer entries and even diagonal entries. The
form $Q$ may be associated to an even lattice $L$ by setting
$L = \Z^{n}$ and defining an inner product on $L$ via
\[
  \langle \vec{x}, \vec{y} \rangle = Q(\vec{x} + \vec{y})
  - Q(\vec{x}) - Q(\vec{y}).
\]
Note that $\langle \vec{x}, \vec{x} \rangle = 2Q(\vec{x})$.
Let
$L' = \{ \vec{x} \in \R^{n} : \langle \vec{x}, \vec{y} \rangle \in \Z
\text{ for all } \vec{y} \in L \}$
denote the dual lattice of $L$. The discriminant form is the finite
abelian group $D = L'/L$ and this group is equipped with an inner
product $D \times D \to \Q/\Z$. For $\vec{x}, \vec{y} \in D$
we will denote the inner product by $\vec{x} \vec{y}$, and
following this convention we denote $\vec{x}^{2}/2 = Q(\vec{x}) \in \Q/\Z$.
We denote by $D^{c} = \{ c\delta : \delta \in D \}$ and
$D_{c} = \{ \delta \in D : c\delta = 0 \}$.

Throughout, if $n$ is a positive integer, we let $d(n)$ denote the
number of divisors of $n$. For a positive-definite integer-valued $Q$, let
$r_{Q}(n) = \#\{ \vec{x} \in \Z^{r} : Q(\vec{x}) = n \}$ and define
\[
  \theta_{Q}(z) = \sum_{n=0}^{\infty} r_{Q}(n) q^{n}, \quad q = e^{2 \pi i z}
\]
to be the theta series of $Q$. When $r$ is even, Theorem 10.9 of
\cite{Iwa} shows that $\theta_{Q}(z)$ is a holomorphic modular form of
weight $r/2$ for the congruence subgroup $\Gamma_{0}(N)$ with
Nebentypus $\chi_{Q} = \chi_{(-1)^{r/2} D(Q)}$. We denote this space
of modular forms $M_{r/2}(\Gamma_{0}(N), \chi_{Q})$ and
$S_{r/2}(\Gamma_{0}(N),\chi_{Q})$ the subspace of cusp forms.  Here
and throughout, $\chi_{D}$ denotes the Kronecker character of the
field $\Q(\sqrt{D})$. We may decompose $\theta_{Q}(z)$ as
\[
  \theta_{Q}(z) = E(z) + C(z)
\]
where $E(z) = \sum_{n=0}^{\infty} a_{E}(n) q^{n}$ is an Eisenstein series,
and $C(z) = \sum_{n=1}^{\infty} a_{C}(n) q^{n}$ is a cusp form.

Let $\Z_{p}$ be the ring of $p$-adic integers. The form $Q$ locally
represents the non-negative integer $m$ if and only if for all primes
$p$ there is a vector $\vec{x}_{p} \in \Z_{p}^{r}$ so that
$Q(\vec{x}_{p}) = m$. We say that $m$ is represented by $Q$ if there
is a vector $\vec{x} \in \Z^{r}$ with $Q(\vec{x}) = m$. We say that
$Q$ is anisotropic at the prime $p$ if
$\vec{x} \in \Z_{p}^{r} \text{ and } Q(\vec{x}) = \vec{0}$ implies
that $\vec{x} = \vec{0}$. This is known to occur only if $r \leq 4$.
Define the genus of $Q$, $\Gen(Q)$, to be the set of quadratic forms
equivalent to $Q$ over $\Z_{p}$ for all primes $p$ (and equivalent to
$Q$ over $\R$). It follows from work of Siegel \cite{Siegel} that
\begin{equation}
\label{genusaverage}
  E(z) = \frac{\sum_{R \in \Gen(Q)} \frac{\theta_{R}(z)}{\# \Aut(R)}}
{\sum_{R \in \Gen(Q)} \frac{1}{\# \Aut(R)}}
  = \sum_{m=0}^{\infty} \left(\prod_{p \leq \infty} \beta_{p}(Q;m)\right) q^{m},
\end{equation}
where 
\[
  \beta_{p}(Q;m) = \lim_{k \to \infty} \frac{\# \{ \vec{x} \in (\Z/p^{k} \Z)^{r} : Q(\vec{x})
\equiv m \pmod{p^{k}}}{p^{(r-1)k}}
\]
is the usual local density. 

If $f(z) = \sum_{n=1}^{\infty} a(n) q^{n}$, define
$f(z) | U(d) = \sum_{n=1}^{\infty} a(dn) q^{n}$ and
$f(z) | V(d) = \sum_{n=1}^{\infty} a(n) q^{dn}$.  An oldform in
$S_{k}(\Gamma_{0}(N), \chi)$ is a cusp form in the span of the images
of $V(e) : S_{k}(\Gamma_{0}(N/d), \chi) \to S_{k}(\Gamma_{0}(N), \chi)$
for $e | d$ and $d > 1$. If $f$ is a modular form of weight $k$ and $\gamma = \begin{bmatrix} a & b \\ c & d \end{bmatrix} \in \SL_{2}(\Z)$,
define
\[
  f |_{k} \gamma = (cz+d)^{-k} f\left(\frac{az+b}{cz+d}\right).
\]
The Petersson inner product on
$S_{k}(\Gamma_{0}(N), \chi)$ is defined by
\[
  \langle f, g \rangle = \frac{3}{\pi [\SL_{2}(\Z) : \Gamma_{0}(N)]}
  \iint_{\H / \Gamma_{0}(N)} f(x+iy) \overline{g(x+iy)} y^{k} \, \frac{dx \, dy}{y^{2}}.
\]
(Some authors omit the factor
$\frac{1}{[\SL_{2}(\Z) : \Gamma_{0}(N)]}$.)  A newform
$f(z) = \sum a(n) q^{n}$ is an eigenfunction of all Hecke operators
normalized so that $a(1) = 1$ and lying in the orthogonal complement
of the space of oldforms.  The Deligne bound gives the bound
$|a(n)| \leq d(n) n^{\frac{k-1}{2}}$ on the $n$th coefficient of a
newform. (In the case of $k = 2$, this result was first established by
Eichler, Shimura, and Igusa.) Distinct newforms are orthogonal with
respect to the Petersson inner product.

\section{Eisenstein contribution}
\label{eis}

In this section we give bounds on
\[
  a_{E}(n) = \frac{\pi^{2} n}{\sqrt{D(Q)}}
  \prod_{p~\text{prime}} \beta_{p}(Q;n).
\]
We make use of the reduction maps in the approach that Jonathan Hanke
takes to computing local densities in \cite{Hanke}, as well as the
explicit formulas of Yang \cite{Yang}. If $p \nmid n$ and
$p \nmid N(Q)$, then $\beta_{p}(Q;n) = 1 -
\frac{\chi_{D}(p)}{p^{2}}$.
In the event that $p | n$ and $p \nmid N(Q)$ we have
$\beta_{p}(Q;n) \geq 1 - \frac{1}{p}$.

If $Q$ is a quadratic form and $p$ is a prime, we introduce the quantity
$r_{p}(Q)$ to measure ``how anisotropic'' the form $Q$ is at the prime $p$.
Take a Jordan decomposition of $Q$ in the form
\[
p^{a_{1}} Q_{1} \bot p^{a_{2}} Q_{2} \bot \cdots \bot p^{a_{k}} Q_{k},
\]
where each $Q_{i}$ has unit discriminant and is indecomposible over $\Z_{p}$.
If $p > 2$ this forces $Q_{i}$ to have dimension $1$, while if $p = 2$,
we must allow $Q_{i}$ to be equivalent to $xy$ or $x^{2} + xy + y^{2}$. Given a
vector $\vec{x}$, we decompose
$\vec{x} = \vec{x}_{1} \bot \vec{x}_{2} \bot \cdots \bot \vec{x}_{k}$,
and we define
\[
  r_{p}(Q) = \inf_{\substack{i, \vec{x} \in \Z_{p}^{r} \\ 1 \leq i \leq k, Q(\vec{x}) = 0}}
  \ord_{p}(a_{i}) + \ord_{p}(\vec{x}_{i}).
\]
We declare $r_{p}(Q) = \infty$ if $Q$ is anisotropic at $p$. We have
$r_{p}(Q) = 0$ if and only if $Q$ is isotropic modulo $p$. It is easy to
see that if $r_{p}(Q)$ is finite, then $r_{p}(Q) \leq \ord_{p}(N(Q))$. The
main result of this section gives three different bounds on $\beta_{p}(n)$.

\begin{lem}
\label{localden}
Suppose that $Q$ is a primitive positive-definite quaternary form, $n$ is locally represented by $Q$ and $r_{p}(Q)$ is defined as above. 
\begin{enumerate}
\item If $n$ satisfies the strong local solubitility condition for $Q$,
then 
\[
  \beta_{p}(n) \geq \begin{cases}
   1 - \frac{1}{p} & \text{ if } p > 2\\
  \frac{1}{4} & \text{ if } p = 2.
\end{cases}
\]
\item If $n$ is primitively locally represented by $Q$, then
\[
  \beta_{p}(n)
  \geq \begin{cases}
    p^{-\lfloor \ord_{p}(D(Q))/2\rfloor} (1-1/p) & \text{ if } p > 2\\
    \frac{1}{16} \cdot 2^{-\lfloor (\ord_{2}(D(Q)) + 1)/2 \rfloor} & \text{ if } p = 2.
\end{cases}
\]
\item In general,
\[
  \beta_{p}(n) \geq
  \begin{cases}
  \left(1 - \frac{1}{p}\right) p^{-\min\{ r_{p}(Q), \ord_{p}(n) \}} & \text{ if } p > 2\\
   2^{-1-\min\{ r_{2}(Q), \ord_{2}(n)\}} & \text{ if } p = 2.
\end{cases}
\]
\end{enumerate}
\end{lem}

\begin{rem}
A somewhat more detailed analysis shows that when $n$ is primitively
locally represented, one has
$\beta_{p}(n) \geq p^{-\lceil \ord_{p}(N(Q))/2 \rceil} (1-1/p)$. In some cases,
this bound is superior to the one above, but in the situation
where $Q$ is equivalent over $\Z_{p}$ to $x^{2} - \epsilon y^{2} +
pz^{2} + p\epsilon w^{2}$, where $\legen{\epsilon}{p} = -1$,
the bounds are the same for $n$ with $\ord_{p}(n) = 1$.
\end{rem}

Before we begin the proof of Lemma~\ref{localden},
we review Hanke's recursive methods to compute $\beta_{p}(n)$.
To handle the cases that $p | N(Q)$, Hanke divides the solutions to
the congruence $Q(\vec{x}) \equiv n \pmod{p^{k}}$ into four classes:
good, zero, bad type I, and bad type II.

We consider a Jordan decomposition of $Q$ in the form described above.
One says that a solution
$\vec{x} = \vec{x}_{1} \bot \vec{x}_{2} \bot \cdots \bot \vec{x}_{k}$
is of \emph{good type} if there is an $i$ so that
$p^{a_{i}} \vec{x}_{i} \not\equiv 0 \pmod{p}$. The integer $n$ satisfies
the strong local solubitility condition if $n$ is locally represented
and there are good type solutions. A solution is of \emph{zero type}
if $\vec{x} \equiv 0 \pmod{p}$, and a solution is of \emph{bad type} if $\vec{x}_{i} \equiv 0 \pmod{p}$ for all $i$ with $a_{i} = 0$,
but there is some $i$ with $\vec{x}_{i} \not\equiv 0 \pmod{p}$ and $a_{i} \geq 1$. A
bad type solution has \emph{type I} if there is some $i$ so that $a_{i} = 1$
and $\vec{x}_{i} \not\equiv 0 \pmod{p}$, and has \emph{type II} if for all
$i$ with $\vec{x}_{i} \not\equiv 0 \pmod{p}$, $a_{i} \geq 2$.

Hanke gives recursive formulas for the local densities by giving
recursive ways to compute good, zero, and bad type solutions. The good
type contribution to the density,
$\beta_{p}^{{\rm Good}}(Q;n) = \beta_{p}^{{\rm Good}}(Q;n \bmod p)$ if
$p > 2$ (and for $p = 2$ we have
$\beta_{p}^{{\rm Good}}(Q;n) = \beta_{p}^{{\rm Good}}(Q;n \bmod 8)$).

The zero type contribution to the density is
$\beta_{p}^{{\rm Zero}}(Q;n) = \frac{1}{p^{2}} \beta_{p}(Q,n/p^{2})$.

In the case of the bad type
solutions, these are related to local densities of $n/p$ and $n/p^{2}$ for different forms. In particular, write $Q = Q_{1} \bot pQ_{2} \bot p^{2} Q_{3}$
where $Q_{1}$, $Q_{2}$ and $Q_{3}$ are (not necessarily indecomposable) forms whose Jordan decompositions have coefficients that are units in $\Z_{p}$.
The bad type I reduction map gives
\[
  \beta_{p}^{{\rm Bad}}(Q,n) \geq p^{1 - \dim Q_{1}} \beta_{p}(pQ_{1} \bot
Q_{2} \bot p^{2} Q_{3}; n/p),
\]
and the bad type II reduction map gives
\[
  \beta_{p}^{{\rm Bad}}(Q,n) \geq p^{\dim Q_{3} - 2} \beta_{p}(Q_{1} \bot
pQ_{2} \bot Q_{3}, n/p^{2}).
\]

\begin{proof}[Proof of Lemma~\ref{localden}]
We first give a lower bound for the contribution from the good type solutions,
assuming that some exist. This will prove part 1 of Lemma~\ref{localden}.

Since $Q$ is primitive, $\dim Q_{1} \geq 1$.  First, we consider the
case that $p \nmid n$, and in this case all solutions will be good
type solutions. If $p \nmid 2D(Q)$, the table on page 363 of
\cite{Hanke} shows that $\beta_{p}(n) \geq 1 -
\frac{1}{p^{2}}$.
Otherwise, diagonalize $Q$ and let $Q_{1}$ be the orthogonal summand
consisting of terms whose coefficients are coprime to $p$. Then,
Yang's formula gives
\[
  \beta_{p}(Q,n) = 1 + \legen{-1}{p}^{\lfloor \frac{\dim Q_{1}}{2} \rfloor}
  \legen{\det(Q_{1})}{p} \cdot p^{1 - \frac{1}{2} \dim Q_{1}} \cdot f_{1}(n)
\]
where
\[
  f_{1}(n) = \begin{cases}
    -\frac{1}{p} & \text{ if } \dim Q_{1} \text{ is even }\\
    \legen{n}{p} \cdot \frac{1}{\sqrt{p}} & \text{ if } \dim Q_{1} \text{ is odd. }
\end{cases}
\]
It follows from this that if $\dim Q_{1} = 2$, then $\beta_{p}(n) \geq 1-1/p$.
If $\dim Q_{1} = 1$ then $\beta_{p}(n) = 1 + \legen{n \det(Q_{1})}{p} \geq 2$ (since we assume that $\beta_{p}(n) > 0$),
and if $\dim Q_{1} = 3$ then $\beta_{p}(n) = 1 + \legen{-\alpha \det(Q_{1})}{p} \cdot \frac{1}{p} \geq 1 - \frac{1}{p}$.

In the case that $p = 2$, we simply enumerate all possibilities for
the $2$-adic Jordan composition of $Q$ mod $8$ and compute the local
density for each $Q$ and each $n \in \{ 1, 3, 5, 7 \}$. We find that
the minimum positive local density is $1/2$. 

Now suppose that $\ord_{p}(n) \geq 1$. When we count good type
solutions for $p > 2$, these exist if and only if $Q_{1}$ is isotropic
mod $p$, which implies that $\dim Q_{1} \geq 2$. If $p \nmid D(Q)$ and
so $\dim Q_{1} = 4$ then we have
$\beta_{p}(Q;n) \geq 1 - \frac{1}{p}$.  If $\dim Q_{1} = 3$, then
$Q_{1} = 0$ defines a conic in $\mathbb{P}^{2}$ with a point on it,
and the variety $Q_{1} = 0$ over $\F_{p}$ is isomorphic to
$\mathbb{P}^{1}$ and so there are $p+1$ points on $Q_{1} = 0$ over
$\mathbb{P}^{2}$. This gives $p(p-1)$ good type solutions mod $p$ and
so $\beta_{p}^{{\rm Good}}(Q;n) \geq 1 - \frac{1}{p}$. If
$\dim Q_{1} = 2$, then there are $p-1$ nonzero vectors mod $p$ for
which $Q_{1}(\vec{x}) = 0$ and we can choose any vector in
$(\Z/p\Z)^{4}$ whose two components correspond to that of
$\vec{x}$. We get $p^{2} (p-1)$ good type solutions and so
$\beta_{p}^{{\rm Good}}(n) \geq 1 - \frac{1}{p}$ provided any good
type solutions exist. An exhaustive enumeration of all forms over
$\Z_{2}$ modulo $8$ shows that if any good type solutions exist, then
$\beta_{p}^{{\rm Good}}(n) \geq 1/2$ if $2 \nmid n$ and
$\beta_{p}^{{\rm Good}}(n) \geq 1/4$ if $2 | n$. This establishes part
1.

For part 2, good type solutions exist if $p \nmid n$ and
this gives $\beta_{p}(n) \geq 1 - 1/p$ and $\beta_{2}(n) \geq 1/2$.
If $p | n$ and there are no good type solutions, there must be bad type I or bad type II solutions. We decompose $Q = Q_{1} \bot pQ_{2} \bot p^{2} Q_{3}$. Then in the case of bad type I solutions, we get
\[
  \beta_{p}(Q;n) \geq p^{1 - \dim Q_{1}} \beta_{p}(pQ_{1} \bot Q_{2} \bot p^{2} Q_{3}; n/p).
\]
However, in the bad type I case, we have solutions $\vec{x} = \vec{x}_{1} \bot \vec{x}_{2} \bot \vec{x}_{3}$
with $\vec{x}_{2} \not\equiv 0 \pmod{p}$, which gives good type solutions for $Q' = pQ_{1} \bot Q_{2} \bot p^{2} Q_{3}$.
Hence, $\beta_{p}(pQ_{1} \bot Q_{2} \bot p^{2} Q_{3};n/p) \geq 1-1/p$ if $p > 2$ and $1/4$ if $p = 2$.
If $\dim Q_{1} = 1$ or $2$, we get $\beta_{p}(Q;n) \geq \frac{1}{p} (1-1/p)$ for $p > 2$. If $\dim Q_{1} = 3$
and $p > 2$, then there are good type solutions, and so $\beta_{p}(Q;n) \geq 1-1/p$. If $p = 2$, we have $\beta_{p}(Q';n/p) \geq 1/4$ and so $\beta_{p}(Q;n) \geq 1/16$.

In the case of bad type II reduction, we prove the claim by induction on $\ord_{p}(D(Q))$. There must be good type solutions
or bad type I solutions when $\ord_{p}(D(Q)) = 0$ or $1$. If $\ord_{p}(D(Q)) \geq 2$ and there are no bad type I solutions, then we have
\[
  \beta_{p}(Q;n) \geq p^{\dim Q_{3} - 2} \beta_{p}(Q_{1} \bot pQ_{2} \bot Q_{3};n/p^{2}).
\]
If we let $Q'' = Q_{1} \bot pQ_{2} \bot Q_{3}$, we have $\ord_{p}(D(Q'')) = \ord_{p}(D(Q)) - 2$ and so we get
\[
  \beta_{p}(Q;n) \geq p^{-1} p^{-\lfloor \ord_{p}(D(Q''))/2 \rfloor} (1-1/p) = p^{-\lfloor \ord_{p}(D(Q))/2\rfloor} (1-1/p)
\]
when $p > 2$ and when $p = 2$ we get
\[
  \beta_{p}(Q;n) \geq \frac{1}{16} \cdot 2^{-\lfloor (\ord_{2}(D(Q))+1)/2 \rfloor},
\]
as desired. This proves part 2.

Now we turn to part 3 of the lemma. The claim already follows from
part 1 if there are any good type solutions. For the zero and bad type
solutions, we prove by induction that if $\beta_{p}(Q;n) > 0$, then
$\beta_{p}(Q;n) \geq (1-1/p) p^{-\ord_{p}(n)}$. We have
$\beta_{p}^{\rm Zero}(Q;n) = \frac{1}{p^{2}} \beta_{p}(Q;n/p^{2})$ and
so if
$\beta_{p}(Q;n/p^{2}) \geq (1-\frac{1}{p}) p^{-\ord_{p}(n/p^{2})}$,
then
\[
  \beta_{p}^{{\rm Zero}}(Q;n) \geq \frac{1}{p^{2}} \cdot
  \left(1 - \frac{1}{p}\right) p^{-\ord_{p}(n/p^{2})} = \left(1-\frac{1}{p}\right) p^{-\ord_{p}(n)}.
\]

For the bad type I solutions we have
\[
  \beta_{p}^{{\rm Bad}}(Q;n) \geq p^{1 - \dim Q_{1}} \beta_{p}(Q';n/p).
\]
This doesn't give the desired bound only when $\dim Q_{1} = 3$. If $\dim Q_{1} = 3$ and $p > 2$, then $Q_{1}$ is isotropic and there are good type solutions
and hence $\beta_{p}(Q;n) \geq 1-1/p$. In the case that $p = 2$,
we have that $Q_{1}$ is either isomorphic to $xy \bot az^{2}$ or
one of the diagonal forms $x^{2} + y^{2} + z^{2}$,
$x^{2} + 3y^{2} + 3z^{2}$, $x^{2} + y^{2} + 5z^{2}$, $x^{2} + 3y^{2} + 7z^{2}$,
$x^{2} + y^{2} + 3z^{2}$, $3x^{2} + 3y^{2} + 3z^{2}$, $x^{2} + y^{2} + 7z^{2}$,
or $3x^{2} + 3y^{2} + 7z^{2}$ (by the classification of $2$-adic forms in Jones \cite{Jones}). A bad type I solution in this case must be a solution $Q(\vec{x}) = n$ with $Q \cong Q_{1} \bot 2Q_{2}$ and $\vec{x}_{1} \equiv 0 \pmod{2}$ but $\vec{x}_{2} \not\equiv 0 \pmod{2}$. In this situation, it must be that $n \equiv 2 \pmod{4}$. However, it is easy to see that each of the possible forms listed
above for $Q_{1}$ must have $\beta_{2}^{{\rm Good}}(Q_{1},n) > 0$ and so
in this case we have $\beta_{2}(Q,n) \geq 1/4$, which proves the desired
result.

For the bad type II solutions we have
\[
  \beta_{p}^{{\rm Bad}}(Q;n) \geq p^{\dim Q_{3} - 2} \beta_{p}(Q'',n/p^{2}).
\]
Since $\dim Q_{3} \leq 3$, we always get the desired result. In all cases,
this proves that $\beta_{p}(Q;n) \geq (1-1/p) p^{-\ord_{p}(n)}$ provided $n$
is locally represented by $Q$.

Now, if $Q = p^{a_{1}} Q_{1} \bot p^{a_{2}} Q_{2} \bot \cdots \bot p^{a_{k}} Q_{k}$ is not anisotropic over $\Q_{p}$, we prove that $\beta_{p}(Q;n) \geq (1-1/p) p^{-r_{p}(Q)}$ (for $p > 2$) and $\beta_{2}(Q;n) \geq (1-1/2) 2^{-r_{2}(Q)-1}$ (for $p = 2$) for all $Q$ that are isotropic over $\Q_{p}$ and locally represent $n$, 
by induction on $r_{p}(Q)$.

In the case that $r_{p}(Q) = 0$, there are good type solutions and the desired
result holds. 

Suppose that $r_{p}(Q) = 1$, $\ord_{p}(n) \geq 1$ and that there are no good type solutions. Then we can decompose $Q = Q_{1} \bot pQ_{2} \bot p^{2} Q_{3}$ and there is a vector $\vec{x}$ with $Q(\vec{x}) = 0$
with $\vec{x} = \vec{x}_{1} \bot \vec{x}_{2} \bot \vec{x}_{3}$ and
$\vec{x}_{2} \not\equiv 0 \pmod{p}$. This implies that
there are bad type I solutions, and the reduction map gives
$\beta_{p}(Q;n) \geq p^{1 - \dim Q_{1}} \beta_{p}(Q';n/p)$ where
$Q' = pQ_{1} \bot Q_{2} \bot p Q_{3}$. The form $Q'$ has $r_{p}(Q') = 0$ and
this gives $\beta_{p}(Q;n) \geq p^{1 - \dim Q_{1}} \beta_{p}(Q';n/p)$.
Moreover, $Q_{2}$ is isotropic modulo $p$ and so $\dim Q_{1} \leq 2$. Thus,
$\beta_{p}(Q;n) \geq (1-1/p) \cdot \frac{1}{p}$, as desired.

Suppose that $r_{p}(Q) \geq 2$ and $n$ is locally represented. If there
are good type solutions, the desired result holds. If there are bad type solutions, then let $Q = Q_{1} \bot pQ_{2} \bot p^{2} Q_{3}$. If there are bad type I solutions, we get
\[
  \beta_{p}(Q;n) \geq p^{1 - \dim Q_{1}} \beta_{p}(Q';n/p)
\]
where $Q' = pQ_{1} \bot Q_{2} \bot pQ_{3}$. We have $r_{p}(Q') = r_{p}(Q) - 1$ and
so $\beta_{p}(Q';n/p) \geq (1-1/p) p^{-r_{p}(Q')}$. If $\dim Q_{1} \leq 2$,
then we get $\beta_{p}(Q;n) \geq (1-1/p) p^{-r_{p}(Q)}$. If $\dim Q_{1} = 3$,
then $p = 2$, $Q_{1}$ is anisotropic, $\ord_{2}(n) = 1$ and
all solutions to $Q' \equiv n/2 \pmod{8}$ are good type solutions.
Hence, $\beta_{2}(Q';n/2) \geq 1/2$ and we get the desired result.

If there are no bad type I solutions, bad type II solutions give
\[
  \beta_{p}(Q;n) \geq p^{\dim Q_{3} - 2} \beta_{p}(Q'';n/p^{2}),
\]
where $Q'' = Q_{1} \bot pQ_{2} \bot Q_{3}$ and $r_{p}(Q'') = r_{p}(Q) - 2$. We have
$\dim Q_{3} \geq 1$ and so if $p > 2$ we get
\[
\beta_{p}(Q;n) \geq p^{-1} (1-1/p) p^{-(r_{p}(Q) - 2)} \geq (1-1/p) p^{1-r_{p}(Q)} \geq (1-1/p) p^{-r_{p}(Q)},
\]
while if $p = 2$ we get
\[
\beta_{2}(Q;n) \geq 2^{-1} (1-1/2) 2^{-(r_{2}(Q) - 1)} \geq (1-1/2) 2^{-r_{2}(Q) - 1}.
\]

Finally, consider the case when there are no good type or bad type solutions.
This implies that $Q(\vec{x}) = n$ implies that $\vec{x} \equiv 0 \pmod{p}$
and this means that $\ord_{p}(n) \leq r_{p}(Q)$. The bound proven
above that $\beta_{p}(n;Q) \geq (1-1/p) p^{-\ord_{p}(n)}$ then implies
that $\beta_{p}(n;Q) \geq (1-1/p) p^{-r_{p}(Q)}$. This proves part 3 of the lemma
by induction.
\end{proof}

%In the case that $r \geq 5$, the results of Browning and Dietmann (see \cite{BrowningDietmann}, Section 5) show the following.
%\begin{thm}
%If $r \geq 5$ and $n$ is locally represented by $Q$, then
%\[
%  \prod_{p~\text{prime}} \beta_{p}(Q;n) \gg D(Q)^{-1/(r-4) - \epsilon}.
%\]
%\end{thm}

%The Hanke paper gives something like the following.
%Let $(n)_{\rm Iso}$ be the largest divisor of $n$ that is not a multiple of
%any anisotropic prime, and let $\chi_{Q}(n) = \legen{D(Q)}{n}$ be the usual
%Kronecker character. Then
%\[
%  a_{E}(n) \geq C_{Q} (n)_{\rm Iso} \prod_{\substack{p | n, p \nmid N(Q) \\
%\chi_{Q}(p) = -1}} \frac{p-1}{p+1}.
%\]
%(The MathOverflow post gives a reference to the work of Gogishvili who has writ%ten several papers about
%the ``optimal'' bound on the Eisenstein coefficients of 4-variable quadratic fo%rms. It might be interesting to try
%to get and read some of these papers.)

\section{Petersson norm bound}
\label{pet}

As above, let $Q$ be a positive-definite, primitive, quaternary
quadratic form, and decompose $\theta_{Q}(z) = E(z) + C(z)$. The goal
of this section is to prove the following.

\begin{thm}
\label{petbound}
For all $\epsilon > 0$, we have
\[
  \langle C, C \rangle \ll \max \{ (N(Q)^{2} D(Q))^{1/4 + \epsilon}, N(Q)^{1 + \epsilon} \}. 
\]
\end{thm}
To do this, we will bound the Petersson norm of
$\theta_{Q}(z) - \theta_{R}(z)$ where $Q$ and $R$ are two quadratic
forms that are in the same genus. We have that
\[
  C = \theta_{Q} - E(z) = \sum_{R \in \Gen(Q)} e_{R} (\theta_{Q} - \theta_{R})
\]
where $e_{R} = \frac{1}{m \# \Aut(R)}$ and $m = \sum_{S \in \Gen(Q)} \frac{1}{\# \Aut(S)}$ is the mass of the genus.
Now, using the triangle inequality, we have that
\[
  \langle C, C \rangle^{1/2} \leq \sum_{R \in \Gen(Q)} \langle e_{R} (\theta_{Q} - \theta_{R}), e_{R} (\theta_{Q} - \theta_{R})\rangle^{1/2} 
= \sum_{R \in \Gen(Q)} e_{R} \langle \theta_{Q} - \theta_{R}, \theta_{Q} - \theta_{R} \rangle^{1/2}.
\]
If $\langle \theta_{Q} - \theta_{R}, \theta_{Q} - \theta_{R} \rangle \leq B$ for all $R \in \Gen(Q)$, it follows that $\langle C, C \rangle^{1/2} \leq \sum_{R} e_{R} B^{1/2} = B^{1/2}$. Hence, it suffices to bound $\langle \theta_{Q} - \theta_{R}, \theta_{Q} - \theta_{R} \rangle$.

Let $f = \theta_{Q} - \theta_{R}$. The Petersson inner product of $f$ with itself is
\[
  \frac{3}{\pi [\SL_{2}(\Z) : \Gamma_{0}(N)]} \iint_{\mathbb{H}/\Gamma_{0}(N)} |f(x+iy)|^{2} \, dx \, dy.
\]
We can write this as
\[
  \frac{3}{\pi [\SL_{2}(\Z) : \Gamma_{0}(N)]}
  \sum_{\gamma \in \Gamma_{0}(N)/\SL_{2}(\Z)}
  \int_{-1/2}^{1/2} \int_{\sqrt{1-x^{2}}}^{\infty} \left|f |_{\gamma} (x+iy)\right|^{2} \, dy \, dx.
\]

We can decompose this as a sum over cusps $a/c$ (letting $w$ denote the width of the cusp,
so $w = \frac{N}{\gcd(c^{2},N)}$) and get
\[
  \frac{3}{\pi [\SL_{2}(\Z) : \Gamma_{0}(N)]}
  \sum_{a/c}
  \int_{-w/2}^{w/2} \int_{\sqrt{1-\{x\}^{2}}}^{\infty} \left|f |_{\gamma} (x+iy)\right|^{2} \, dx \, dy.
\]
We lower the boundary on $y$ to $\sqrt{3}/2$ and get
\tiny
\[
  \frac{3}{\pi [\SL_{2}(\Z) : \Gamma_{0}(N)]}
  \sum_{a/c} \int_{-w/2}^{w/2} \int_{\sqrt{3}/2}^{\infty}
  \left(\sum_{n=1}^{\infty} a_{a/c}(n)
  e^{2 \pi i nx/w} e^{-2 \pi ny/w} \right) \left(\sum_{m=1}^{\infty} \overline{a_{a/c}(m)} e^{-2 \pi i mx/w} e^{-2 \pi my/w}\right) \, dy \, dx.
\]
\normalsize
Here $f |_{\gamma} = \sum_{n=1}^{\infty} a_{a/c}(n) q^{n/w}$.
We get cancellation unless $n = m$ and this gives us
\begin{align*}
  \langle f, f \rangle &\leq \frac{3}{\pi [\SL_{2}(\Z) : \Gamma_{0}(N)]}
  \sum_{a/c} w \int_{\sqrt{3}/2}^{\infty}
  \sum_{n=1}^{\infty} |a_{a/c}(n)|^{2} e^{-4 \pi ny/w} \, dy\\ 
  &= \frac{3}{4 \pi^{2} [\SL_{2}(\Z) : \Gamma_{0}(N)]}
  \sum_{a/c} w \sum_{n=1}^{\infty}
  \frac{|a_{a/c}(n)|^{2}}{n/w} e^{-2 \pi \sqrt{3} n/w}.
\end{align*}
At the cost of a factor of $2$ we can replace $a_{a/c}(n)$ with the coefficient of
$\theta_{S}$ at the cusp $a/c$, for some $S \in \Gen(Q)$.

In \cite{Scheithauer}, Scheithauer works out an explicit formula for
the Weil representation attached to a quadratic form, which yields
explicit formulas for the coefficients of $\theta_{S}$ at all
cusps. To state this formula, we need some notation.  Let
$S = \frac{1}{2} \vec{x}^{T} A \vec{x}$ be a positive-definite
quadratic form and let $L$ be the corresponding lattice. Let
$D = L'/L$. For a positive integer $c$ define $\phi : D \to D$ given
by $\phi(x) = cx$. Let $D^{c}$ be the image of $\phi$ and $D_{c}$ be
the kernel of $\phi$.  Define
$D^{c*} = \{ \alpha \in D : c\gamma^{2}/2 + \alpha\gamma \equiv 0
\pmod{1} \text{ for all } \gamma \in D_{c} \}$.
By Proposition 2.1 of \cite{Scheithauer}, $D^{c*}$ is a coset of
$D^{c}$, and $2 D^{c*} = D^{c}$.

Theorem 4.7 of \cite{Scheithauer} gives that the
coefficient of $q^{n/w}$ in the Fourier expansion of $\theta_{S} |_{\gamma}$ is
\[
  \overline{\xi} \frac{1}{\sqrt{|D^{c}|}}
  \sum_{\beta \in D^{c*}} \xi_{\beta}
  \# \{ \vec{v} \in \bigcup L + \beta : \beta \in D^{c*}, S(\vec{v}) = n/w \},
\]
where $\xi$ and $\xi_{\beta}$ are roots of unity. Define $T = \{ \vec{x} \in L' : \vec{x} \bmod L \in L'/L \subseteq D^{c} \cup D^{c*} \}$. This
is a lattice that contains $L$ and all the vectors in the count above are
contained in $T$. Let $R$ be $4w$ times the quadratic form
corresponding to $T$. 
\begin{lem}
The form $R$ is an integral quadratic form.
\end{lem}
It follows that the coefficient of $q^{n/w}$ in $\theta_{S}|_{\gamma}$
is bounded by $\frac{1}{\sqrt{|D^{c}|}} r_{R}(4n)$. 

\begin{proof}
First, we claim that if $\alpha \in D^{c*}$, then $2 \alpha
\in D^{c}$. This is because if $\alpha \in D^{c*}$, then the definition
gives that for all $\delta \in D_{c}$ we have
$c \delta^{2}/2 + \alpha \delta \equiv 0 \pmod{1}$
which implies that $c \delta^{2} + 2 \alpha \delta \equiv 0 \pmod{1}$,
or $\langle \delta, c\delta \rangle + \langle 2 \alpha, \delta \rangle \equiv 0 \pmod{1}$. Since $\delta \in D_{c}$, $c\delta = 0$ and
so $\langle \delta, c\delta \rangle = 0$. Thus $\langle 2\alpha, \delta \rangle = 0$ so $2 \alpha \in (D_{c})^{\perp} = D^{c}$. 

Now, if $\vec{x} \in D^{c}$, then
$\vec{x} = \vec{a} + c\vec{b}$ where $\vec{a} \in L$
and $\vec{b} \in L'$. Then we have
\begin{align*}
  \frac{1}{2} \langle \vec{x}, \vec{x} \rangle
  &= \frac{1}{2} \langle \vec{a}, \vec{a} \rangle
  + c \langle \vec{a}, \vec{b} \rangle + \frac{c^{2}}{2} \langle \vec{b}, \vec{b} \rangle.
\end{align*}
The first term is an integer, because $L$ is an even lattice.
The second term is an integer (in fact, a multiple of $c$)
because $\vec{b} \in L'$. For the third term,
if we write $w = \frac{N}{\gcd(N,c^{2})} = \frac{N}{c^{2}/m}$, then
we get $\frac{wc^{2} \langle \vec{b}, \vec{b} \rangle}{2} = \frac{N m \langle \vec{b}, \vec{b} \rangle}{2}$. From the definition,
we know that $N \langle \vec{b}, \vec{b} \rangle/2 \in \Z$ from the definition
of the level and so $w \langle \vec{x}, \vec{x} \rangle/2 \in \Z$.

Hence, if $\vec{x} \in D^{c}$, then $w S(\vec{x}) \in \Z$, while
if $\vec{x} \in D^{c*}$, then $2 \vec{x} \in D^{c}$ and so
$4w S(\vec{x}) = w S(2 \vec{x}) \in \Z$. This proves that $R$ is an integral
quadratic form. 
\end{proof}

We have that the discriminant of $R$ is $\frac{(4w)^{4} |D|}{|D^{c}|^{2}}$ or $\frac{4^{3} w^{4} |D|}{|D^{c}|^{2}}$ depending
on whether $D^{c*} \ne D^{c}$ or $D^{c*} = D^{c}$.

Putting together the computations above we see that
\[
  \langle C, C \rangle \ll \frac{1}{[\SL_{2}(\Z) : \Gamma_{0}(N)]}
  \sum_{a/c} w \sum_{n=1}^{\infty} \frac{r_{R}(4n)^{2}}{|D^{c}| (n/w)} e^{-2 \pi \sqrt{3} n/w}.
\]
We rewrite this expression using
\[
  \frac{e^{-2 \pi \sqrt{3} n/w}}{n/w} = \int_{n}^{\infty} \left(\frac{1}{x(x/w)} + \frac{2 \pi \sqrt{3}}{x}\right) e^{-2 \pi \sqrt{3} (x/w)} \, dx
\]
and get
\[
  \langle C, C \rangle
  \ll \frac{1}{[\SL_{2}(\Z) : \Gamma_{0}(N)]}
  \sum_{a/c} \int_{1}^{\infty} \left(\sum_{n \leq 4x} r_{R}(n)^{2}\right)
   \cdot \left(\frac{1}{(x/w)^{2}} + \frac{2 \pi \sqrt{3}}{x/w}\right) e^{-2 \pi \sqrt{3} (x/w)} \, dx.
\]
To estimate this expression, we now need bounds on $\sum_{n \leq x} r_{R}(n)^{2}$.
We use that
\[
  \sum_{n \leq x} r_{R}(n)^{2} \leq
  \left(\sum_{n \leq x} r_{R}(n)\right) \cdot \left(\max_{n \leq x} r_{R}(n)\right).
\]
We use the bound
\[
  \sum_{n \leq x} r_{R}(n) \ll \frac{x^{2}}{D(R)^{1/2}}
  + O\left(\frac{x^{3/2}}{D(R)^{1/2} \lambda_{\rm max}^{-1/2}}\right)
  \ll \frac{x^{2}}{D(R)^{1/2}} + x^{3/2}.
\]
which is given in the proof of Lemma 4.1 of \cite{Blomer}. Here $\lambda_{\rm max}$ is the largest eigenvalue of the Gram matrix of $R$.

To bound $\max_{n \leq x} r_{R}(n)$, we follow the argument given in
the answer to MathOverflow question 256576 (by GH from MO). We assume
that $R$ is Korkin-Zolotarev reduced. This implies that we can write
\[
  R = a_{1} x^{2} + a_{2} (y + m_{12} x)^{2} + a_{3} (z + m_{13} x + m_{23} y)^{3} + a_{4} (w + m_{12} x + m_{13} y + m_{14} z)^{2},
\]
where $a_{1} a_{2} a_{3} a_{4} = \frac{1}{16} D(R)$ and
$a_{i} \geq \frac{3}{4} a_{i+1}$ and $|m_{ij}| \leq 1/2$. (See Theorem
2.1.1 from \cite{Kitaoka}.)

We write
\[
  r_{R}(n) = \sum_{z} \sum_{w} \# \{ (x,y) \in \Z^{2} :
  R(x,y,z,w) = n \}.
\]
There are at most $\left(2 \sqrt{\frac{n}{a_{3}}} + 1 \right) \left(2 \sqrt{\frac{n}{a_{4}}} + 1\right)$ choices of the pair $(z,w)$. For each such choice,
$R(x,y,z,w)$ is a binary quadratic polynomial (with coefficients
that are polynomial in $n$). The proof of Lemma 8 in
\cite{BlomerPohl} implies that if $P = ax^{2} + bxy + cy^{2} + dx + ey + f$,
with $b^{2} - 4ac < 0$, then number of solutions to $P(x,y) = 0$ is at most
$6 d(|4a (b^{2} - 4ac)(4acf + bd^{3} - ae^{2} - b^{2} f - cd^{2})|)$.
Using that $d(n) = O_{\epsilon}(n^{\epsilon})$ gives that
\[
  r_{R}(n) \ll \left(2 \sqrt{\frac{n}{a_{3}}} + 1 \right) \left(2 \sqrt{\frac{n}{a_{4}}} + 1\right) D(R)^{\epsilon} n^{\epsilon}.
\]
We have that $a_{3} a_{4} \geq \frac{D(R)}{16 a_{1} a_{2}} \gg D(R)^{1/2}$,
and this gives $r_{R}(n) \ll n^{1+\epsilon} D(R)^{-1/4+\epsilon} + n^{1/2}$.
Combining this bound with the one for $\sum_{n \leq x} r_{R}(n)^{2}$, we get
\[
  \sum_{n \leq x} r_{R}(n)^{2} \ll \frac{x^{3 + \epsilon}}{D(R)^{3/4 - \epsilon}} + \frac{x^{5/2 + \epsilon}}{D(R)^{1/4 - \epsilon}} + x^{2}.
\]
We plug this into the bound on $\langle C, C \rangle$ and get
\small
\[
  \langle C, C \rangle \ll \frac{1}{[\SL_{2}(\Z) : \Gamma_{0}(N)]}
  \sum_{a/c} \frac{1}{|D^{c}|} \int_{1}^{\infty} \left(\frac{x^{3 + \epsilon}}{D(R)^{3/4 - \epsilon}} + \frac{x^{5/2 + \epsilon}}{D(R)^{1/4 - \epsilon}} + x^{2}\right)
\left(\frac{1}{x/w} + \frac{1}{(x/w)^{2}}\right) e^{-2 \pi \sqrt{3} (x/w)} \, dx.
\]
\normalsize
Now, we set $u = x/w$, so $du = dx/w$. We use that
$\int_{1/w}^{\infty} u^{\alpha} e^{-2 \pi \sqrt{3} u} \, du \ll 1$ for $\alpha > -1/2$.
This gives
\[
  \langle C, C \rangle
  \ll \frac{1}{[\SL_{2}(\Z) : \Gamma_{0}(N)]}
  \sum_{a/c} \frac{w}{|D^{c}|} \left(w^{3+\epsilon} D(R)^{-3/4 + \epsilon}
  + w^{5/2+\epsilon} D(R)^{-1/4 + \epsilon} + w^{2}\right).
\]

Since $D(R) \gg \frac{w^{4} |D|}{|D^{c}|^{2}}$ we get
\[
  \langle C, C \rangle
  \ll \frac{1}{[\SL_{2}(\Z) : \Gamma_{0}(N)]}
  \sum_{a/c} \frac{w^{1+\epsilon} |D^{c}|^{1/2 - \epsilon}}{|D|^{3/4 - \epsilon}}
  + \frac{w^{5/2 + \epsilon}}{|D|^{1/4 + \epsilon} |D^{c}|^{1/2 - \epsilon}}
  + \frac{w^{3}}{|D^{c}|} .
\]
The first term above is bounded by
$N^{\epsilon} \frac{1}{|D|^{1/4-\epsilon}} \cdot \frac{1}{[\SL_{2}(\Z) : \Gamma_{0}(N)]} \sum_{a/c} w = \frac{N^{\epsilon}}{|D|^{1/4-\epsilon}}$. 

is much smaller than the third,
as $|D^{c}| \leq |D|$.

\begin{lem}
\label{sumbound}
We have
\[
  \frac{1}{[\SL_{2}(\Z) : \Gamma_{0}(N)]}
  \sum_{a/c} \frac{w^{2}}{|D^{c}|} \ll |D|^{\epsilon}.
\]
\end{lem}

The lemma implies that the second term above
\[
\frac{1}{[\SL_{2}(\Z) : \Gamma_{0}(N)]}
\sum_{a/c} \frac{w^{5/2 + \epsilon}}{|D|^{1/4 + \epsilon} |D^{c}|^{1/2 - \epsilon}}
\ll \frac{1}{[\SL_{2}(\Z) : \Gamma_{0}(N)]}
\sum_{a/c} \frac{w^{2} N^{1/2 + \epsilon} |D|^{1/4}}{|D^{c}|}
\ll N^{1/2 + \epsilon} |D|^{1/4 + \epsilon},
\]
since $w \leq N$. The third term above is at most
\[
\frac{1}{[\SL_{2}(\Z) : \Gamma_{0}(N)]}
\sum_{a/c} \frac{w^{3}}{|D^{c}|}
\ll \frac{1}{[\SL_{2}(\Z) : \Gamma_{0}(N)]}
\sum_{a/c} \frac{w^{2} N}{|D^{c}|} \ll N |D|^{\epsilon}.
\]
It follows from this that $\langle C, C \rangle \ll \max \{ (N(Q)^{2} |D(Q)|)^{1/4 + \epsilon}, N(Q)^{1+\epsilon} \}$. This reduces Theorem~\ref{petbound} to Lemma~\ref{sumbound}.

\begin{proof}[Proof of Lemma~\ref{sumbound}]
For a given divisor $c$ of $N$, the number of cusps $a/c$ with
denominator $c$ is $\phi(\gcd(c,N/c))$, where $\phi$ is Euler's
totient. Moreover, by the first isomorphism theorem, we have that $|D|
= |D^{c}| |D_{c}|$.  We write
\begin{align*}
\frac{1}{[\SL_{2}(\Z) : \Gamma_{0}(N)]} \sum_{a/c} \frac{w^{2}}{|D^{c}|}
&= \frac{1}{[\SL_{2}(\Z) : \Gamma_{0}(N)]} \sum_{c | N}
\frac{N^{2} (w/N)^{2}}{|D|/|D_{c}|} \cdot \phi(\gcd(c,N/c))\\
&= \frac{N^{2}}{|D| [\SL_{2}(\Z) : \Gamma_{0}(N)]}
  \sum_{c | N} |D_{c}| (w/N)^{2} \phi(\gcd(c,N/c)).
\end{align*}
Recalling that $w = \frac{N}{\gcd(c^{2},N)}$ we see that
$w/N = \frac{1}{\gcd(c^{2},N)} = \frac{1}{c} \cdot \frac{1}{\gcd(c,N/c)}$.
Define $g(c) = |D_{c}| \frac{\phi(\gcd(c,N/c))}{c^{2} \gcd(c,N/c)^{2}}$. It is not hard to see that $g(c)$ is a multiplicative function of $c$ and hence
$f(k) = \sum_{c | k} g(c)$ is multiplicative.

Let $p$ be a prime divisor of $|D|$. We will assume that $p > 2$ and
at the end indicate briefly the modifications that must occur if
$p = 2$. If $p > 2$, then the fact that $Q$ is primitive implies that
the $p$-Sylow subgroup of $D$ is
$\Z/p^{a_{1}} \Z \times \Z/p^{a_{2}} \Z \times \Z/p^{a_{3}} \Z$ where
$0 \leq a_{1} \leq a_{2} \leq a_{3}$. We have $\ord_{p}(N) = a_{3}$
and $\ord_{p}(|D|) = a_{1} + a_{2} + a_{3}$.  This gives \scriptsize
\[
  f(p^{a_{3}}) =
  1 + \sum_{i=1}^{a_{1}} \frac{p^{3i} p^{\min\{ i-1, a_{3} - i-1\}} (p-1)}{p^{2i}
  p^{\min \{2i, 2a_{3} - 2i\}}}
  + \sum_{i=a_{1}+1}^{a_{2}} \frac{p^{2i} p^{\min\{ i-1, a_{3} - i - 1\}} (p-1)}{p^{2i} p^{\min \{2i, 2a_{3} - 2i\}}} +
  \sum_{i=a_{2}+1}^{a_{3}} \frac{p^{i} p^{\min\{ i-1, a_{3} - i - 1\}} (p-1)}{p^{2i} p^{\min \{2i, 2a_{3} - 2i\}}}.
\]
\normalsize
We rewrite this as
\[
  1 + \frac{p-1}{p} \left[
  \sum_{i=1}^{a_{1}} p^{\max \{ 2i - a_{3}, 0\}}
  + \sum_{i=a_{1}+1}^{a_{2}} p^{\max \{-i, i - a_{3} \}}
  + \sum_{i=a_{2}+1}^{a_{3}} p^{\max \{ -2i, -a_{3} \}}\right].
\]
The second sum is largest when $a_{2} = a_{3}$
and we get $p^{a_{1} + 1 - a_{3}} + p^{a_{1} + 2 - a_{3}} + \cdots + p^{-1} + 1
\leq \frac{1}{1 - 1/p}$. The third sum is at most $p^{-2} + p^{-4} + \cdots
= \frac{1}{p^{2} - 1}$.

When $2a_{1} \leq a_{3}$, the first sum is $a_{1}$. When $2a_{1} > a_{3}$,
the largest term in the first sum is $p^{2a_{1} - a_{3}}$ and each
other term is smaller by a factor of at least $p$. Thus, the first
sum is bounded by $p^{2a_{1} - a_{3}} \cdot \left(1 + \frac{a_{1} - 1}{p}\right)$.
We get then that
\[
  \frac{1}{[\SL_{2}(\Z) : \Gamma_{0}(N)]} \sum_{a/c}
  \frac{w^{2}}{|D^{c}|}
  = \prod_{p | N} h(p),
\]
where 
\[
  h(p) = \begin{cases}
  \frac{p}{p+1} p^{-a_{1}-a_{2}} \left(1 + \frac{p-1}{p}
  \left[ a_{1} + \frac{p}{p-1} + \frac{1}{p^{2} - 1}\right]\right)
  & \text{ if } 2a_{1} \leq a_{3}\\
  \frac{p}{p+1} p^{-a_{1} - a_{2}} \left(1 +
  \frac{p-1}{p} \left[ p^{2a_{1} - a_{3}} \cdot \left(1 + \frac{a_{1} - 1}{p}\right)
  + \frac{p}{p-1} + \frac{1}{p^{2} - 1}\right]\right) & \text{ if } 2a_{1} > a_{3}.
\end{cases}
\]
A somewhat tedious analysis of the cases shows that for all $p \geq 3$,
$h(p) \leq 2$. This gives
\[
  \frac{1}{[\SL_{2}(\Z) : \Gamma_{0}(N)]}
  \sum_{c | N} \frac{w^{2}}{|D^{c}|} \phi(\gcd(c,N/c))
  \leq \prod_{p | N} 2 = 2^{\omega(N)} \ll |D|^{\epsilon}.
\]
(If $p = 2$, then we can
have $(\Z/2\Z)^{4} \subseteq D$, but we cannot have $(\Z/4\Z)^{4} \subseteq D$.
This means that we can have the Sylow $2$-subgroup of $D$ is isomorphic
to $\Z/2^{a_{1}} \Z \times \Z/2^{a_{2}} \Z \times \Z/2^{a_{3}} \Z \times \Z/2^{a_{4}} \Z$, but $0 \leq a_{1} \leq 1$. The argument above requires no change except
replacing $a_{3}$ with $a_{4}$ in the event that $a_{1} = 0$. If $a_{1} = 1$,
we add a fourth summation consisting of one term, and the size of that term
is at most $4$. This only affects the size of the implied constant.)
\end{proof}

\section{Proof of main results}
\label{last}

In this section we prove Theorem~\ref{main}. First, we rely on
the preprint \cite{RSPAY} of Rainer Schulze-Pillot and
Abdullah Yenirce to translate the bound on $\langle C, C \rangle$
given above into a bound on the Fourier coefficient $|a_{C}(n)|$. The weight
$2$ case of Theorem~10 of \cite{RSPAY} gives that
\[
  |a_{C}(n)| \leq 4 \pi e^{4 \pi} \left(\langle C, C \rangle
\dim S_{2}(\Gamma_{0}(N),\chi)\right)^{1/2} d(n) \sqrt{n} N^{1/2}
\prod_{p | N} \frac{\left(1 + \frac{1}{p}\right)^{1/3}}{\sqrt{1 - \frac{1}{p^{4}}}}.
\]
This can be restated as
$|a_{C}(n)| \ll \langle C, C \rangle^{1/2} N^{1 + \epsilon} d(n)
\sqrt{n}$.
We use this bound in the second, third, and fourth parts of the main
theorem.

For the first part of the main theorem, we need a more detailed analysis
of the work of Schulze-Pillot and Yenirce. It is well-known that
$S_{2}(\Gamma_{0}(N),\chi)$ is spanned by newforms $f$ of level $M | N$
with charater $\chi$ and their ``translates'', that is, the images
under the operator $V(d)$ (for $d$ dividing $N/M$). Given a newform
$f$ of level $M | N$ and a prime $p$ dividing $N/M$, let
$\tilde{f} = \frac{f}{\sqrt{\langle f,f \rangle}}$. Theorem~8 of
\cite{RSPAY} gives an explicit orthogonal basis $\{ g_{0}, g_{1}, \ldots, g_{r} \}$ for the space $W_{p}(f)$, the space spanned by $f, f|V(p), \ldots, 
f|V(p^{r})$. The only basis elements that involve $f$ are $g_{0} = \tilde{f}$,
\[
  g_{1} = \begin{cases}
    p \tilde{f} | V(p) - \frac{\overline{a(p)}}{p} \tilde{f}
  & \text{ if } p \nmid N\\
   p \tilde{f} | V(p) - \frac{\overline{a(p)}}{p +1} \tilde{f} & \text{ if } p | N, \end{cases}
\]
and
\[
  g_{2} = p^{2} \tilde{f} | V(p^{2}) - \overline{a(p)} \tilde{f} | V(p)
  + \frac{\overline{\chi(p)}}{p} \tilde{f}
\]
and then only if $p | N$. Here $a(p)$ is the eigenvalue of the $p$th Hecke
operator acting on $f$, which is also the $p$th Fourier coefficient
of $f$. Moreover,
$\langle g_{i}, g_{i} \rangle = 1 + O(p^{-1/2})$.

In general, the space of ``translates'' of the newform $f$ is isometric
to the tensor product of the spaces $W_{p_{i}}(f)$ for the $p_{i} | M/N$
with the isometry mapping
\[
  \tilde{f}|V(p_{1}^{r_{1}}) \otimes
  \tilde{f}|V(p_{2}^{r_{2}}) \otimes \cdots
  \otimes \tilde{f}|V(p_{z}^{r_{z}}) \to
  \tilde{f} | V(p_{1}^{r_{1}} p_{2}^{r_{2}} \cdots p_{z}^{r_{z}}).
\]
(This is given on the bottom of page 4 and the top of page 5 of \cite{RSPAY}.)
As a consequence, there is an orthonormal basis for $W_{p}(f)$ with the property
that if $\gcd(n,N) = 1$, then the $n$th Fourier coefficient of every
basis element is $\ll \frac{d(n) \sqrt{n}}{\langle f, f \rangle} (1 + 1/p^{1/2})$, because if $\gcd(n,N) = 1$, no contribution is made to the $n$th coefficient from $f | V(p^{i})$ if $i \geq 1$. Using the tensor product isometry, we come
to the conclusion that if $C(z) = \sum c_{i} h_{i}$
is an expression for $C$ in terms of the orthonormal basis constructed
by Schulze-Pillot and Yenirce, then for $\gcd(n,N) = 1$, the $n$th coefficient
of $C(z)$ is bounded by 
\[
  |a_{C}(n)| \ll \frac{\left(\sum |c_{i}|\right) d(n) \sqrt{n} \prod_{p | N} 
(1+1/p^{1/2})}{\min_{f} \langle f, f \rangle}.
\]
Iwaniec and Kowalski (\cite{IK}, Corollary 5.45, page 140) give
the lower bound $\langle f, f \rangle \gg N^{-\epsilon}$, which is ineffective.
Using that $\langle C, C \rangle = \sum |c_{i}|^{2}$ and the Cauchy-Schwarz inequality
\[
  \sum |c_{i}| \leq \left(\sum |c_{i}|^{2}\right)^{1/2}
\left(\dim S_{2}(\Gamma_{0}(N,\chi)\right)^{1/2}
\]
we get that $|a_{C}(n)| \ll d(n) \sqrt{n} N^{1/2 + \epsilon} \langle C, C \rangle^{1/2}$ if $\gcd(n,N) = 1$.

\begin{proof}[Proof~of~Theorem~\ref{main}]
In the case that $\gcd(n,N(Q)) = 1$ and $n$ is locally represented,
we have that $\beta_{p}(n) \geq 1-1/p$ and $\beta_{2}(n) \gg 1$. This gives
$a_{E}(n) \gg n^{1-\epsilon} D(Q)^{-1/2}$. Theorem~\ref{petbound} gives
\[
  \langle C, C \rangle \ll \max \{ N(Q)^{1/2+\epsilon} D(Q)^{1/4 + \epsilon},
N(Q)^{1 + \epsilon} \}
\]
and so $|a_{C}(n)| \ll d(n) \sqrt{n} \max \{ N(Q)^{3/4 + \epsilon} D(Q)^{1/8 + \epsilon}, N(Q)^{1 + \epsilon} \}$. It follows that $a_{E}(n) > |a_{C}(n)|$ if
$n \gg \max \{ N(Q)^{3/2+\epsilon} D(Q)^{5/4 + \epsilon}, N(Q)^{2+\epsilon} D(Q)^{1+\epsilon} \}$.

In the case that $n$ satisfies the strong local solubility condition,
then Lemma~\ref{localden} again gives that $\beta_{p}(n) \geq 1-1/p$ and $\beta_{2}(n) \geq 1/4$ and so $a_{E}(n) \gg n^{1 - \epsilon} D(Q)^{-1/2}$.
We use the general bound $|a_{C}(n)| \ll \langle C, C \rangle^{1/2} N^{1+\epsilon} d(n) \sqrt{n}$ which gives
\[
  |a_{C}(n)| \ll \max \{ N(Q)^{5/4 + \epsilon} D(Q)^{1/8 + \epsilon},
  N(Q)^{3/2 + \epsilon} \} d(n) \sqrt{n}.
\]
It follows that $a_{E}(n) > |a_{C}(n)|$ if $n \gg \max \{ N(Q)^{5/2 + \epsilon} D(Q)^{5/4 + \epsilon}, N(Q)^{3 + \epsilon} D(Q)^{1 + \epsilon} \}$. 

In the case that $n$ is primitively represented by $Q$, Lemma~\ref{localden}
gives $\beta_{p}(n) \gg p^{-\ord_{p}(D(Q)/2)} (1-1/p)$
for $p | 2D(Q)$. We have $\beta_{p}(n) \geq 1-1/p$ for $p \nmid 2D(Q)$
which gives $a_{E}(n) \gg n^{1-\epsilon} D(Q)^{-1-\epsilon}$ provided $n$ is primitively locally represented by $Q$. We use the same bound on $|a_{C}(n)|$ as
in the previous case and this gives $a_{E}(n) > |a_{C}(n)|$ if $n
\gg \max \{ N(Q)^{5/2 + \epsilon} D(Q)^{9/4 + \epsilon},
N(Q)^{3 + \epsilon} D(Q)^{2 + \epsilon} \}$. 

Finally, Lemma~\ref{localden} implies that if $n$ is locally represented
by $Q$, then either $\beta_{p}(n)
\geq \left(1-\frac{1}{p}\right) p^{-\ord_{p}(N(Q))}$ or $\ord_{p}(n) > \ord_{p}(N(Q))$. This latter case can only occur if $r_{p}(Q) = \infty$, namely when
$Q$ is anisotropic at $p$. 

If $n$ is locally represented by $Q$ and
$\beta_{p}(n) \geq (1-1/p) p^{-\ord_{p}(N(Q))}$ for all $p$, then we
have $a_{E}(n) \gg \frac{n}{\sqrt{D(Q)} N(Q)^{1+\epsilon}}$.  Using
that
$|a_{C}(n)| \ll \max \{ N(Q)^{5/4 + \epsilon} D(Q)^{1/8 + \epsilon},
N(Q)^{3/2 + \epsilon} \} d(n) \sqrt{n}$
we get that $n$ is represented by $Q$ if
$n \gg \max \{ N(Q)^{9/2 + \epsilon} D(Q)^{5/4 + \epsilon}, N(Q)^{5 +
  \epsilon} D(Q)^{1 + \epsilon} \}$.

Hence, if $n$ is larger than this bound, locally represented,
but not represented by $Q$, then we must have
$\ord_{p}(n) > \ord_{p}(N(Q))$ for some prime $p | N(Q)$ that is an anisotropic
prime for $Q$. This implies that $n$ is
$p$-stable in the terminology of Hanke (see Definition~3.6 of \cite{Hanke}) and
Corollary 3.8.2 implies that $r_{Q}(n) = r_{Q}(np^{2k})$ for any $k \geq 1$.
Hence, $p^{2} | n$ and $r_{Q}(np^{2k}) = 0$ for all $k \geq 1$, as desired.
\end{proof}

\bibliographystyle{amsplain}
\bibliography{qfrefs}

\providecommand{\bysame}{\leavevmode\hbox to3em{\hrulefill}\thinspace}
\providecommand{\MR}{\relax\ifhmode\unskip\space\fi MR }
% \MRhref is called by the amsart/book/proc definition of \MR.
\providecommand{\MRhref}[2]{%
  \href{http://www.ams.org/mathscinet-getitem?mr=#1}{#2}
}
\providecommand{\href}[2]{#2}
\begin{thebibliography}{10}

\bibitem{REU2016}
M.~Barowsky, W.~Damron, A.~Mejia, F.~Saia, N.~Schock, and K.~Thomspon,
  \emph{Classically integral quadratic forms excepting at most two values},
  Proceedings of the American Mathematical Society, To appear.

\bibitem{Bhar}
M.~Bhargava, \emph{On the {C}onway-{S}chneeberger fifteen theorem}, Quadratic
  forms and their applications ({D}ublin, 1999), Contemp. Math., vol. 272,
  Amer. Math. Soc., Providence, RI, 2000, pp.~27--37. \MR{1803359
  (2001m:11050)}

\bibitem{BH}
M.~Bhargava and J.~Hanke, \emph{Universal quadratic forms and the
  290-{T}heorem}, Preprint.

\bibitem{Blomer}
Valentin Blomer, \emph{Uniform bounds for {F}ourier coefficients of
  theta-series with arithmetic applications}, Acta Arith. \textbf{114} (2004),
  no.~1, 1--21. \MR{2067869}

\bibitem{BlomerPohl}
Valentin Blomer and Anke Pohl, \emph{The sup-norm problem on the {S}iegel
  modular space of rank two}, Amer. J. Math. \textbf{138} (2016), no.~4,
  999--1027. \MR{3538149}

\bibitem{BrowningDietmann}
T.~D. Browning and R.~Dietmann, \emph{On the representation of integers by
  quadratic forms}, Proc. Lond. Math. Soc. (3) \textbf{96} (2008), no.~2,
  389--416. \MR{2396125 (2009f:11035)}

\bibitem{Coprime3}
J.~DeBenedetto and J.~Rouse, \emph{Quadratic forms representing all integers
  coprime to 3}, The Ramanujan Journal, To appear.

\bibitem{Fom}
O.~M. Fomenko, \emph{Estimates of {P}etersson's inner product with an
  application to the theory of quaternary quadratic forms}, Dokl. Akad. Nauk
  SSSR \textbf{152} (1963), 559--562. \MR{0158874 (28 \#2096)}

\bibitem{Fom2}
\bysame, \emph{Estimates for scalar squares of cusp forms, and arithmetic
  applications}, Zap. Nauchn. Sem. Leningrad. Otdel. Mat. Inst. Steklov. (LOMI)
  \textbf{168} (1988), no.~Anal. Teor. Chisel i Teor. Funktsii. 9, 158--179,
  190. \MR{982491 (90c:11037)}

\bibitem{Hanke}
J.~Hanke, \emph{Local densities and explicit bounds for representability by a
  quadratric form}, Duke Math. J. \textbf{124} (2004), no.~2, 351--388.
  \MR{2079252 (2005m:11060)}

\bibitem{HsiaIcaza}
J.~S. Hsia and M.~I. Icaza, \emph{Effective version of {T}artakowsky's
  theorem}, Acta Arith. \textbf{89} (1999), no.~3, 235--253. \MR{1691853}

\bibitem{IK}
H.~Iwaniec and E.~Kowalski, \emph{Analytic number theory}, American
  Mathematical Society Colloquium Publications, vol.~53, American Mathematical
  Society, Providence, RI, 2004. \MR{2061214 (2005h:11005)}

\bibitem{Iwa}
Henryk Iwaniec, \emph{Topics in classical automorphic forms}, Graduate Studies
  in Mathematics, vol.~17, American Mathematical Society, Providence, RI, 1997.
  \MR{1474964}

\bibitem{Jones}
Burton~W. Jones, \emph{A canonical quadratic form for the ring of 2-adic
  integers}, Duke Math. J. \textbf{11} (1944), 715--727. \MR{0012640}

\bibitem{Kap}
I.~Kaplansky, \emph{Ternary positive quadratic forms that represent all odd
  positive integers}, Acta Arith. \textbf{70} (1995), no.~3, 209--214.
  \MR{1322563 (96b:11052)}

\bibitem{Kap7}
Irving Kaplansky, \emph{The first nontrivial genus of positive definite ternary
  forms}, Math. Comp. \textbf{64} (1995), no.~209, 341--345. \MR{1265017}

\bibitem{Kitaoka}
Y.~Kitaoka, \emph{Arithmetic of quadratic forms}, Cambridge Tracts in
  Mathematics, vol. 106, Cambridge University Press, Cambridge, 1993.
  \MR{1245266 (95c:11044)}

\bibitem{Kloosterman}
H.~D. Kloosterman, \emph{On the representation of numbers in the form
  {$ax^2+by^2+cz^2+dt^2$}}, Acta Math. \textbf{49} (1927), no.~3-4, 407--464.
  \MR{1555249}

\bibitem{Reinke}
T.~Reinke, \emph{Darstellung durch {S}pinorgeschlechtern tern\"arer
  quadratischer formen}, 2002, Dissertation, M\"unster.

\bibitem{451}
Jeremy Rouse, \emph{Quadratic forms representing all odd positive integers},
  Amer. J. Math. \textbf{136} (2014), no.~6, 1693--1745. \MR{3282985}

\bibitem{Scheithauer}
Nils~R. Scheithauer, \emph{The {W}eil representation of {${\rm SL}_2(\Bbb Z)$}
  and some applications}, Int. Math. Res. Not. IMRN (2009), no.~8, 1488--1545.
  \MR{2496771}

\bibitem{SP}
R.~Schulze-Pillot, \emph{On explicit versions of {T}artakovski's theorem},
  Arch. Math. (Basel) \textbf{77} (2001), no.~2, 129--137. \MR{1842088
  (2002i:11036)}

\bibitem{RSPAY}
R.~Schulze-Pillot and A.~Yenirce, \emph{{P}etersson products of bases of spaces
  of cusp forms and estimates for {F}ourier coefficients}, Preprint.

\bibitem{Siegel}
C.~L. Siegel, \emph{\"{U}ber die analytische {T}heorie der quadratischen
  {F}ormen}, Ann. of Math. (2) \textbf{36} (1935), no.~3, 527--606.
  \MR{1503238}

\bibitem{Tart}
W.~Tartakowsky, \emph{Die {G}esamtheit der {Z}ahlen, die durch eine positive
  quadratische {F}orm $f(x_{1},\ldots,,x_{s})(s⩾4)$ darstellbar sind}, Izv.
  Akad. Nauk. SSSR \textbf{7} (1929), 111--122, 165--195.

\bibitem{Watson2}
G.~L. Watson, \emph{Integral quadratic forms}, Cambridge Tracts in Mathematics
  and Mathematical Physics, No. 51, Cambridge University Press, New York, 1960.
  \MR{0118704}

\bibitem{Watson}
\bysame, \emph{Quadratic {D}iophantine equations}, Philos. Trans. Roy. Soc.
  London Ser. A \textbf{253} (1960/1961), 227--254. \MR{0130211}

\bibitem{Yang}
T.~Yang, \emph{An explicit formula for local densities of quadratic forms}, J.
  Number Theory \textbf{72} (1998), no.~2, 309--356. \MR{1651696 (99j:11034)}

\end{thebibliography}

\end{document}